\documentclass[11pt,leqno]{amsart}
\usepackage{amsmath,amsthm,amssymb}
\usepackage{tabularx}
\newcommand{\GF}{{\mathbb F}}

\newcommand{\R}{{\mathbb R}}

\usepackage{comment}

\newcommand{\wt}{{\rm wt}}
\newcommand{\supp}{{\rm supp}}

\DeclareMathOperator{\Harm}{Harm}

\usepackage{color} 
\usepackage{url} 

\newtheorem{Thm}{Theorem}[section]
\newtheorem{Lem}[Thm]{Lemma}

\theoremstyle{definition}
\newtheorem{Def}[Thm]{Definition}
\newtheorem{Rem}[Thm]{Remark}

\newtheorem{Conj}[Thm]{Conjecture}

\newcommand{\RR}{\mathbb{R}}
\newcommand{\ZZ}{\mathbb{Z}}

\newcommand{\NN}{\mathbb{N}}

\newcommand{\FF}{\mathbb{F}}

\makeatletter
\@addtoreset{equation}{section}

\makeatother

\begin{document}

\title[A note on the Assmus--Mattson theorem for some binary codes II]{
A note on the Assmus--Mattson theorem for some binary codes II
}

\author{Eiichi Bannai}
\address{Faculty of Mathematics, 
Kyushu University (emeritus), Japan}
\email{bannai@math.kyushu-u.ac.jp} 
\author{Tsuyoshi Miezaki}
\address{		Faculty of Science and Engineering, 
		Waseda University, 
		Tokyo 169--8555, Japan
}
\email{miezaki@waseda.jp} 
\author{Hiroyuki Nakasora*}
\thanks{*Corresponding author}
\address{Institute for Promotion of Higher Education, Kobe Gakuin University, Kobe
651--2180, Japan}
\email{nakasora@ge.kobegakuin.ac.jp}

\date{}

\maketitle

\begin{abstract}

Let $C$ be a four-weight binary code, which has all one vector. 
Furthermore, we assume that $C$ 
supports $t$-designs for all weights obtained 
from {the Assmus--Mattson theorem}. 
We previously showed that $t\leq 5$. 
In the present paper, 
we show an analogue of this result in the cases of five and six-weight codes. 

\end{abstract}


\noindent
{\small\bfseries Key Words and Phrases.}
Assmus--Mattson theorem, $t$-designs, harmonic weight enumerator.\\ \vspace{-0.15in}

\noindent
2010 {\it Mathematics Subject Classification}.
Primary 05B05;
Secondary 94B05, 20B25.\\ \quad

\setcounter{section}{+0}
\section{Introduction}

Let $D_{w}$ be the support design of a binary code $C$ for weight $w$ and
\begin{align*}
\delta(C)&:=\max\{t\in \mathbb{N}\mid \forall w, 
D_{w} \mbox{ is a } t\mbox{-design}\},\\ 
s(C)&:=\max\{t\in \mathbb{N}\mid \exists w \mbox{ s.t.~} 
D_{w} \mbox{ is a } t\mbox{-design}\}.
\end{align*}
We note that $\delta(C) \leq s(C)$. 
In the previous papers 
\cite{{extremal design H-M-N}, MMN, extremal design2 M-N,MN-tec, dual support designs}, 
we considered the possible occurrence of $\delta(C)<s(C)$.
This was motivated by Lehmer's conjecture, which is 
an analogue of $\delta(C)<s(C)$ in the theory of lattices 
and vertex operator algebras. 
For the details, 
see \cite{{BM1},{BM2},{BMY},{Lehmer},{Miezaki},{Miezaki2},{MMN},{mn-typeI},{Venkov},{Venkov2}}.

Let us explain our results. 
Throughout this paper, $C$ denotes a 
binary $[n,k,d]$ code and $\mathbf{1}_n\in C$. 
Let $C^{\perp}$ be a binary $[n,n-k,d^{\perp}]$ dual code of $C$. We set $C_u:=\{c\in C\mid \wt(c)=u\}$. 
Note that $d^{\perp}$ is even because $\mathbf{1}_n\in C$. 
We always assume that there exists $t\in \NN$ that
satisfies the following condition: 
\begin{align}
d^{\perp}-t=\sharp\{u\mid C_u \neq \emptyset, 0<u\leq n-t\}. \label{eqn:AM}
\end{align}
This is a condition of the Assmus--Mattson theorem 
(see Theorem \ref{thm:assmus-mattson}), and say the AM-condition.
Let $D_{u}$ and $D^{\perp}_{w}$ be the support designs of $C$ and $C^{\perp}$ 
for weights $u$ and $w$, respectively. 
Then, by (\ref{eqn:AM}) and Theorem \ref{thm:assmus-mattson}, 
$D_u$ and $D^{\perp}_w$ are $t$-designs 
(also $s$-designs for $0<s<t$)
for any $u$ and $w$, respectively. 

Let $C$ satisfy the AM-condition.
In the previous paper \cite{dual support designs}, 
for $1\leq d^{\perp}-t\leq 3$, 
we imposed some restrictions on $d^{\perp}$ and $t$.
The main results of the present paper are the following theorems. 
For $4\leq d^{\perp}-t\leq 5$, 
we also impose 
some restrictions on $d^{\perp}$ and $t$.
\begin{Thm}\label{thm:main1}
\begin{enumerate}
\item [{\rm (1)}]
If $C$ satisfies the AM-condition with $d^{\perp}-t=4$, 
then $(d^{\perp},t)=(6,2)$ or $(8,4)$.

\item [{\rm (2)}]
If $C$ satisfies the AM-condition with $d^{\perp}-t=5$, 
then $(d^{\perp},t)=(6,1), (8,3)$, or $(10,5)$.

\end{enumerate}
\end{Thm}

For cases in which $d^{\perp}-t=4$ or $5$,
the following theorem provides a criterion for $n$ and $d$ such that 
$\delta(C^{\perp})<s(C^{\perp})$ occurs.
Let $d=d_{1}$ and $d_{2}$ be the second smallest weight of $C$.

\begin{Thm}\label{thm:main2}
\begin{enumerate}
\item [{\rm (1)}]
Let $C$ satisfy the AM-condition with $(d^{\perp},t)=(6,2), (8,4)$.
Let $w\in \NN$ such that 
\begin{align*} 
&\sum_{i=0}^{w}  (-1)^{w-i} \binom{d_{1}-(t+1)}{w-i}\binom{n-2d_{1}}{2i+1}\\
&-\frac{n-2d_{1}}{n-2d_{2}}\sum_{j=0}^{w}  (-1)^{w-j} \binom{d_{2}-(t+1)}{w-j}\binom{n-2d_{2}}{2j+1}
=0. 
\end{align*}
Then 
$D^{\perp}_{2w+t+2}$ is a $(t+1)$-design. 
Hence, we have $\delta(C^\perp)<s(C^\perp)$. 

\item [{\rm (2)}]
Let $C$ satisfy the AM-condition with $(d^{\perp},t)=(6,1), (8,3)$, or $(10,5)$.
Let $w\in \NN$ such that 
\begin{align*}
&\sum_{i=0}^{w}  (-1)^{w-i} \binom{d_{1}-(t+1)}{w-i}\binom{n-2d_{1}}{2i}  \\
&-\frac{(n-2d_{1})(n-2d_{1}-2)}{(n-2d_{2})(n-2d_{2}-2)} \sum_{j=0}^{w}  (-1)^{w-j} \binom{d_{2}-(t+1)}{w-j}\binom{n-2d_{2}}{2j} \\
&+\frac{8(d_{2}-d_{1})(n-d_{1}-d_{2}-1)}{(n-2d_{2})(n-2d_{2}-2)} (-1)^{w+1} \binom{n/2-(t+1)}{w} 
=0. 
\end{align*}
Then 
$D^{\perp}_{2w+t+1}$ is a $(t+1)$-design. 
Hence, we have $\delta(C^\perp)<s(C^\perp)$.

\end{enumerate}
\end{Thm}


This paper is organized as follows: 
In Section~\ref{sec:pre}, 
we provide background material and terminology.
We review the concept of 
harmonic weight enumerators and some theorems of designs, 
which are used in the proof of the 
main results. 
In Section~\ref{sec:4}, 
we provide proofs of the case $d^{\bot}-t=4$, 
Theorem \ref{thm:main1} (1), and Theorem \ref{thm:main2} (1). 
In Section~\ref{sec:5}, 
we provide proofs of the case $d^{\bot}-t=5$, 
Theorem \ref{thm:main1} (2), and Theorem \ref{thm:main2} (2). 
Finally, in Section~\ref{sec:rem}, 
we conclude the paper with some remarks. 

We performed all the computer calculations in this paper with the help of 
{\sc Magma} \cite{Magma} and 
{\sc Mathematica} \cite{Mathematica}.


\section{Preliminaries}\label{sec:pre}

\subsection{Background material and terminology}\label{sec:terminology}

A binary linear code $C$ of length $n$ is a subspace of $\GF_2^n$. 
An inner product $({x},{y})$ on $\FF_2^n$ is given 
by
\[
(x,y)=\sum_{i=1}^nx_iy_i,
\]
where $x,y\in \FF_2^n$ with $x=(x_1,x_2,\ldots, x_n)$ and 
$y=(y_1,y_2,\ldots, y_n)$. 
The dual of a linear code $C$ is defined as follows: 
\[
C^{\perp}=\{{y}\in \FF_{2}^{n}\ | \ ({x},{y}) =0\ \mbox{ for all }{x}\in C\}.
\]
A linear code $C$ is self-dual 
if $C=C^{\perp}$. 
For $x \in\FF_2^n$,
the weight $\wt(x)$ is the number of its nonzero components. 
The minimum distance of the code $C$ is 
$\min\{\wt( x)\mid  x \in C,  x \neq  0 \}$. 
A linear code of length $n$, dimension $k$, and 
minimum distance $d$ is called an $[n,k,d]$ code 
(or $[n,k]$ code) and 
the dual code is called an $[n,n-k,d^{\perp}]$ code. 
{

A $t$-$(v,k,{\lambda})$ design (or $t$-design, for short) is a pair 
$\mathcal{D}=(X,\mathcal{B})$, where $X$ is a set of points of 
cardinality $v$, and $\mathcal{B}$ is a collection of $k$-element subsets
of $X$ called blocks, with the property that any $t$ points are 
contained in precisely $\lambda$ blocks.


The support of a vector 
${x}:=(x_{1}, \dots, x_{n})$, 
$x_{i} \in \GF_{2}$ is 
the set of indices of its nonzero coordinates: 
${\rm supp} ({x}) = \{ i \mid x_{i} \neq 0 \}$\index{$supp (x)$}.
Let $\Omega:=\{1,\ldots,n\}$ and 
$\mathcal{B}(C_\ell):=\{\supp({x})\mid {x}\in C_\ell\}$. 
Then for a code $C$ of length $n$, 
we say that $C_\ell$ is a $t$-design if 
$\mathcal{D}_\ell=(\Omega,\mathcal{B}(C_\ell))$ is a $t$-design. 
We call $\mathcal{D}_\ell$ a support design of $C$.

The following theorem is from Assmus and Mattson \cite{assmus-mattson}. It is one of the 
most important theorems in coding theory and design theory:
\begin{Thm}[\cite{assmus-mattson}] \label{thm:assmus-mattson}
Let $C$ be a binary $[n,k,d]$ linear code and $C^{\bot}$ be the $[n,n-k,d^{\bot}]$ dual code. 
Let $t$ be an integer less than $d$.  
Let $C$ have at most $d^{\bot}-t$ non-zero weights less than or equal to $n-t$. Then, 
for each weight $u$ with $d \leq u \leq n-t$, the support design in $C$ is a $t$-design, 
and for each weight $w$ with $d^{\bot} \leq w \leq n$, the support design 
in $C^{\bot}$ is a $t$-design.

\end{Thm}

\subsection{Harmonic weight enumerators}\label{sec:weight enumerators}

In this subsection, we review the concept of 
harmonic weight enumerators. 

Let $C$ be a code of length $n$. The weight distribution of the code $C$ 
is the sequence $\{A_{i}\mid i=0,1, \dots, n \}$, where $A_{i}$ is the number of codewords of weight $i$. 
The polynomial
$$W_C(x ,y) = \sum^{n}_{i=0} A_{i} x^{n-i} y^{i}$$
is called the weight enumerator of $C$.
The weight enumerator of the code $C$ and its dual $C^{\perp}$ are related. 
The following theorem, proposed by MacWilliams, is called the MacWilliams identity:
\begin{Thm}[\cite{mac}]\label{thm: macwilliams iden.} 
Let $W_C(x ,y)$ be the weight enumerator of an $[n,k]$ code $C$ over $\GF_{q}$ and let $W_{C^\perp}(x ,y)$ be 
the weight enumerator of the dual code $C^\perp$. Then
$$W_{C^\perp} (x ,y)= q^{-k} W_C(x+(q-1)y,x-y).$$
\end{Thm}

A striking generalization of the MacWilliams identity 
was provided by Bachoc \cite{Bachoc}, 
who proposed the concept of harmonic weight enumerators and 
a generalization of the MacWilliams identity. 
Harmonic weight enumerators have many applications; 
in particular, the relations between coding theory and 
design theory are reinterpreted and progressed by harmonic weight 
enumerators \cite {Bachoc,Bannai-Koike-Shinohara-Tagami}. 
For the reader's convenience, we quote 
the definitions and properties of discrete harmonic functions from \cite{Bachoc,Delsarte}. 

Let $\Omega=\{1, 2,\ldots,n\}$ be a finite set (which is the set of coordinates of the code) and 
let $X$ be the set of its subsets, where for all $h= 0,1, \ldots, n$, $X_{h}$ is the set of its $h$-subsets.
Let $\R X$ and $\R X_h$ denote the free real vector spaces spanned by the elements of $X$ and $X_{h}$, respectively. 
An element of $\R X_h$ is denoted by
$$f=\sum_{z\in X_h}f(z)z$$
and is identified with the real-valued function on $X_{h}$ given by 
$z \mapsto f(z)$. 

Such an element $f\in \R X_h$ can be extended to an element $\widetilde{f}\in \R X$ by setting, for all $u \in X$,
$$\widetilde{f}(u)=\sum_{z\in X_h, z\subset u}f(z).$$
If an element $g \in \R X$ is equal to some $\widetilde{f}$, for $f \in \R X_{h}$, we say that $g$ has degree $h$. 
The differentiation $\gamma$ is the operator defined by linearity from 
$$\gamma(z) =\sum_{y\in X_{h-1},y\subset z}y$$
for all $z\in X_h$ and for all $h=0,1, \ldots n$, and $\Harm_{h}$ is the kernel of $\gamma$:
$$\Harm_h =\ker(\gamma|_{\R X_h}).$$

\begin{Thm}[{{\cite[Theorem 7]{Delsarte}}}]\label{thm:design}
A set $\mathcal{B} \subset X_{m}$, where $m \leq n$ of blocks is a $t$-design 
if and only if $\sum_{b\in \mathcal{B}}\widetilde{f}(b)=0$ 
for all $f\in \Harm_h$, $1\leq h\leq t$. 
\end{Thm}

In \cite{Bachoc}, the harmonic weight enumerator associated with a binary linear code $C$ was defined as follows: 
\begin{Def}
Let $C$ be a binary code of length $n$ and let $f\in\Harm_{h}$. 
The harmonic weight enumerator associated with $C$ and $f$ is

$$W_{C,f}(x,y)=\sum_{{c}\in C}\widetilde{f}({c})x^{n-\wt({c})}y^{\wt({c})}.$$
\end{Def}

Bachoc proved the following MacWilliams-type equality: 
\begin{Thm}[\cite{Bachoc}] \label{thm: Bachoc iden.} 
Let $W_{C,f}(x,y)$ be 
the harmonic weight enumerator associated with the code $C$ 
and the harmonic function $f$ of degree $h$. Then 
$$W_{C,f}(x,y)= (xy)^{h} Z_{C,f}(x,y),$$
where $Z_{C,f}$ is a homogeneous polynomial of degree $n-2h$, and satisfies
$$Z_{C^{\bot},f}(x,y)= (-1)^{h} \frac{2^{n/2}}{|C|} Z_{C,f} \left( \frac{x+y}{\sqrt{2}}, \frac{x-y}{\sqrt{2}} \right).$$
\end{Thm}

\section{Case $d^{\perp}-t=4$}\label{sec:4}
In this section, we always assume that $d^{\perp}-t=4$. 
Then the weights of $C$ are $0,d_{1}, d_{2}, n-d_{2}, n-d_{1}, n$.
Before providing the proof, we show the following lemma, 
which will be used in the proof of Theorem \ref{thm:d^perp>10}: 
\begin{Lem}\label{lem:eq}
Let $n,k\in \ZZ_{\geq 0}$. 
The solutions of the following equation
\begin{align*}
\sum_{i=0}^4\binom{n-1}{i}=
\frac{1}{24}(24 - 18 n + 23 n^2 - 6 n^3 + n^4)=2^k
\end{align*}
are as follows: 
\[
(n,k)=(0,0),(1,0),(2,1),(3,2),(4,3),(5,4),(10,8). 
\]
\end{Lem}

\begin{proof}
We assume that $k\equiv 0\pmod{2}$ and $y=2^{k/2}$. 
Then 
\begin{align*}
&\frac{1}{24}(24 - 18 n + 23 n^2 - 6 n^3 + n^4)=y^2\\
\Leftrightarrow 
&6\times (24 - 18 n + 23 n^2 - 6 n^3 + n^4)=6\times 24\times y^2. 
\end{align*}
Let $Y=12y$. Then 
\[
6\times (24 - 18 n + 23 n^2 - 6 n^3 + n^4)=Y^2. 
\]
By "IntegerQuarticPoints" (a command in {\sc Magma}), 
we obtain the solutions $(n,Y)$: 
\begin{center}
$(-12,456)$,
$(10,-192)$,
 $(5,48)$,
 $(3,-24)$,
 $(36,2928)$,\\
 $(1,12)$,
 $(-2,-36)$,
 $(0,-12)$,
 $(-237,139344)$. 
\end{center}
If 
\[
\frac{1}{24}(24 - 18 n + 23 n^2 - 6 n^3 + n^4)=2^k
\]
is a power of $2$, then 
$n=10,5,3,1,0$, 
$k=8,4,2,0,0$, respectively. 

We assume that $k\equiv 1\pmod{2}$, $y=2^{(k-1)/2}$. 
Then 
\begin{align*}
&\frac{1}{24}(24 - 18 n + 23 n^2 - 6 n^3 + n^4)=2y^2\\
\Leftrightarrow 
&3\times (24 - 18 n + 23 n^2 - 6 n^3 + n^4)=3\times 24\times 2y^2. 
\end{align*}
Let $Y=12y$. Then 
\[
3\times (24 - 18 n + 23 n^2 - 6 n^3 + n^4)=Y^2. 
\]
Let $n=X+2$. Then
\[
(144 +102 X + 33 X^2 +6 X^3 + 3X^4)=Y^2. 
\]
By "IntegerQuarticPoints" (a command in {\sc Magma}), 
we obtain the solutions $(X,Y)$: 
\begin{center}
$(2,-24)$, $(0,-12)$. 
\end{center}
We have $n=4,2$. 
If 
\[
\frac{1}{24}(24 - 18 n + 23 n^2 - 6 n^3 + n^4)=2^k 
\]
is a power of $2$, then 
$n=4,2$, 
$k=3,1$, respectively. 
\end{proof}

\subsection{Proof of Theorem \ref{thm:main1} (1)}

In this subsection, 
we provide the proof of Theorem \ref{thm:main1} (1).
Let 
\[
W_{C} (x,y)=x^{n}+\alpha x^{n-d_{1}}y^{d_{1}}+\beta x^{n-d_{2}}y^{d_{2}}+\beta x^{d_{2}}y^{n-d_{2}}+\alpha x^{d_{1}}y^{n-d_{1}}
+y^{n}
\]
be the weight enumerator of $C$. We remark that 
\[
1+\alpha+\beta=2^{k-1}. 
\]

We show that if $d^\perp\geq 10$, then 
we have the constraints Eqs.~(\ref{eq:A2}).
By Theorem \ref{thm: macwilliams iden.}, 
\begin{align*}
W_{C^\perp} (x ,y)=& 2^{-k} W_C(x+y,x-y) \\
=&2^{-k}\sum_{i\geq 0}A_ix^{n-2i}y^{2i}.
\end{align*}
If $d^\perp\geq 10$, then 
we have the following constraints: 
\begin{align}\label{eq:A2}
A_2=
A_4=
A_6=
A_8=0. 
\end{align}



Using Eqs.~(\ref{eq:A2}), we show the following theorem:
\begin{Thm}\label{thm:d^perp>10}
There is no code $C$ with $d^{\perp}\geq 10$.
\end{Thm}
\begin{proof}
We assume that $C$ has $d^{\perp}\geq 10$. 
Using Eqs.~(\ref{eq:A2}), 
we delete their constant terms as follows: 
\begin{align*}
\binom{n}{4}A_2
-\binom{n}{2} A_4=0
\Leftrightarrow X_{1}\alpha+Y_{1}\beta=0,\\
\binom{n}{6} A_2
-\binom{n}{2} A_6=0
\Leftrightarrow X_{2}\alpha+Y_{2}\beta=0,\\
\binom{n}{8} A_2
-\binom{n}{2} A_8=0
\Leftrightarrow X_{3}\alpha+Y_{3}\beta=0,
\end{align*}
where $X_{i}\ (1\leq i\leq 3)$ and ~$Y_{i}\ (1\leq i\leq 3)$ are 
the coefficients of $\alpha$ and ~$\beta$, respectively. 
Let 
\begin{align*}
M_1=
\begin{pmatrix}
X_1&Y_1\\
X_2&Y_2
\end{pmatrix}, 
M_2=
\begin{pmatrix}
X_1&Y_1\\
X_3&Y_3
\end{pmatrix}. 
\end{align*}
By 
$\alpha\neq 0$ and
$\beta\neq 0$, 
\[
\det(M_1)=0, 
\det(M_2)=0. 
\]
Then using {\sc Mathematica} \cite{Mathematica}, 
we obtain the solutions, which are listed on the homepage of
one of the authors \cite{miezaki}.

Because 
$0<d_1<d_2<n/2$ and $5<n$, 
the solutions (1a)--(20a), (23a), (24a), (27a), and (28a) 
in \cite{miezaki} are impossible. 
We show that (25a) in \cite{miezaki} is impossible. 
The other cases (21a),(22a), and (26a) can be proved similarly. 

Using Eqs.~(\ref{eq:A2}), 
and (25a) in \cite{miezaki}, 
using {\sc Mathematica}, we obtain 
\[
1+\alpha+\beta=\frac{1}{24}(24 - 18 n + 23 n^2 - 6 n^3 + n^4). 
\]
By Lemma \ref{lem:eq}, 
\[(n,k)=(0,0),(1,0),(2,1),(3,2),(4,3),(5,4),(10,8),\] and 
it is clear that these cases 
are impossible. 
\end{proof}

\begin{proof}[Proof of Theorem \ref{thm:main1} (1)]
By Theorem \ref{thm:d^perp>10}, 
$d^\perp \leq 8$. Hence,
$(d^\perp, t) = (6, 2)$ or $(8, 4)$.
\end{proof}


\subsection{Proof of Theorem \ref{thm:main2} (1)}

In this subsection, 
we provide the proof of Theorem \ref{thm:main2} (1). 
\begin{proof}
The harmonic weight enumerator of $f\in \Harm_{t+1}$ is 
\begin{align*}
W_{C,f}
=&a_{1}x^{n-d_{1}}y^{d_{1}}+ a_{2}x^{n-d_{2}}y^{d_{2}} 
+ b_{2}x^{d_{2}}y^{n-d_{2}} +b_{1}x^{d_{1}}y^{n-d_{1}}\\
=(xy)^{t+1}&(a_{1}x^{n-d_{1}-(t+1)}y^{d_{1}-(t+1)}+ a_{2}x^{n-d_{2}-(t+1)}y^{d_{2}-(t+1)} \\
&+b_{2}x^{d_{2}-(t+1)}y^{n-d_{2}-(t+1)} +b_{1}x^{d_{1}-(t+1)}y^{n-d_{1}-(t+1)}), 
\end{align*}
where $a_{1},a_{2},b_{1}, b_{2}\in \RR$. 
We set
\begin{align*}
Z_{C,f}=a_{1}x^{n-d_{1}-(t+1)}y^{d_{1}-(t+1)}+ a_{2}x^{n-d_{2}-(t+1)}y^{d_{2}-(t+1)} \\
+b_{2}x^{d_{2}-(t+1)}y^{n-d_{2}-(t+1)} +b_{1}x^{d_{1}-(t+1)}y^{n-d_{1}-(t+1)}. 
\end{align*}

Then, by Theorem \ref{thm: Bachoc iden.}, 
\begin{align*}
Z_{C^{\perp},f}
= & a'_{1}(x+y)^{n-d_{1}-(t+1)}(x-y)^{d_{1}-(t+1)}\\
&+ a'_{2}(x+y)^{n-d_{2}-(t+1)}(x-y)^{d_{2}-(t+1)} \\
& +b'_{2}(x+y)^{d_{2}-(t+1)}(x-y)^{n-d_{2}-(t+1)}\\
& +b'_{1}(x+y)^{d_{1}-(t+1)}(x-y)^{n-d_{1}-(t+1)}.
\end{align*}
Since the coefficient of $x^{n-2t-2}$ in $Z_{C^{\perp},f}$ is zero, 
$a'_{1}+b'_{1}=0$ and $a'_{2}+b'_{2}=0$.

By $d^{\perp} \neq t+2$, the coefficient of $x^{n-2t-3}y$ in $Z_{C^{\perp},f}$ is zero. 
Then, \begin{align*}
&a'_{1}(n-d_{1}-(t+1)-(d_{1}-(t+1)))\\
&+a'_{2}(n-d_{2}-(t+1)-(d_{2}-(t+1)))=0.
\end{align*}
Hence, 
\[
a'_{2}=-\frac{n-2d_{1}}{n-2d_{2}}a'_{1}.
\] 
Then,
\begin{align}\label{eqn:w3} 
Z_{C^{\perp},f}
= & a'_{1} \big( (x+y)^{n-d_{1}-(t+1)}(x-y)^{d_{1}-(t+1)}\\  \notag
&-\frac{n-2d_{1}}{n-2d_{2}} (x+y)^{n-d_{2}-(t+1)}(x-y)^{d_{2}-(t+1)} \\ \notag
& +\frac{n-2d_{1}}{n-2d_{2}}(x+y)^{d_{2}-(t+1)}(x-y)^{n-d_{2}-(t+1)}\\ \notag
& -(x+y)^{d_{1}-(t+1)}(x-y)^{n-d_{1}-(t+1)} \big) \\ \notag
= & a'_{1} \big( (x^{2}-y^{2})^{d_{1}-(t+1)}(x+y)^{n-2d_{1}}\\ \notag
&-\frac{n-2d_{1}}{n-2d_{2}} (x^{2}-y^{2})^{d_{2}-(t+1)}(x+y)^{n-2d_{2}} \\ \notag
& +\frac{n-2d_{1}}{n-2d_{2}}(x^{2}-y^{2})^{d_{2}-(t+1)}(x-y)^{n-2d_{2}} \\\notag
&-(x^{2}-y^{2})^{d_{1}-(t+1)}(x-y)^{n-2d_{1}} \big). 
\end{align}

Let 
\[
W_{C^\perp,f}=(xy)^{t+1}Z_{C^{\perp},f}=\sum{p_i}x^{n-i}y^i.
\]
Recall that if $p_{2w+t+2}=0$ then $D^{\perp}_{2w+t+2}$ is a $(t+1)$-design. 
By (\ref{eqn:w3}), 
\begin{align*}
p_{2w+t+2}=&{\rm (constant)}\\
\times &
\Bigg( \sum_{i=0}^{w}  (-1)^{w-i} \binom{d_{1}-(t+1)}{w-i}\binom{n-2d_{1}}{2i+1} \\
&-\frac{n-2d_{1}}{n-2d_{2}}\sum_{j=0}^{w}  (-1)^{w-j} \binom{d_{2}-(t+1)}{w-j}\binom{n-2d_{2}}{2j+1} \Bigg).
\end{align*}
By Theorem \ref{thm:design}, 
if the equation 
\begin{align*}
& \sum_{i=0}^{w}  (-1)^{w-i} \binom{d_{1}-(t+1)}{w-i}\binom{n-2d_{1}}{2i+1}\\
&-\frac{n-2d_{1}}{n-2d_{2}}\sum_{j=0}^{w}  (-1)^{w-j} \binom{d_{2}-(t+1)}{w-j}\binom{n-2d_{2}}{2j+1} =0, 
\end{align*}
$D^{\perp}_{2w+t+2}$ is a $(t+1)$-design. 
\end{proof}

\section{Case $d^{\perp}-t=5$}\label{sec:5}
In this section, we always assume that $d^{\perp}-t=5$. 
Then the weights of $C$ are $0,d_1, d_2, n/2, n-d_2, n-d_1, n$.
Before providing the proof, we show the following lemma, 
which will be used in the proof of Theorem \ref{thm:d^perp>12}: 
\begin{Lem}\label{lem:eq2}
Let $n,m\in \ZZ_{\geq 0}$. 
The solutions of the following equation
\begin{align*}
\sum_{i=0}^5\binom{n-1}{i}
=\frac{1}{120}(184 n - 110 n^2 + 55 n^3 - 10 n^4 + n^5)=2^m
\end{align*}
are as follows: 
\[(n,m)=(1,0),(2,1),(3,2),(4,3),(5,4),(6,5),(12,10).\] 
\end{Lem}

\begin{proof}
The following argument was made by 
Professor ~Max Alekseyev 
in \cite{mathoverflow}. 
We seek that the integer solutions of 
\[
S(n,5)=\sum_{i=0}^5\binom{n}{i}=2^m. 
\]
Then, $5!S(n,5)=(n+1)g_5(n)$, 
where 
\[
g_5(n)=120 - 26 n + 31 n^2 - 6 n^3 + n^4. 
\]
Hence, 
\[
n+1=2^k, 3\times 2^k, 5\times 2^k, 15\times 2^k. 
\]
First, we assume that $n+1=2^k$. Then 
\begin{align*}
g_5(2^k-1)&=184 + 31\times 2^{2 k} + 2^{4 k} - 9\times 2^{1 + k} - 2^{2 + k}\\
&- 11 \times 2^{3 + k} 
+  3\times  2^{3 + 2 k} - 3\times  2^{1 + 3 k} - 2^{2 + 3 k}=15\times 2^\ell. 
\end{align*}
We note that $184=2^3\times 23$.

If $k\leq 5$, then 
we obtain the solution $(n,\ell)=(0,3),(1,3),(3,4)$. 
If $k>5$, then considering modulo $2^7$, 
the following equation has no solutions: 
\begin{align*}
g_5(2^k-1)&=184 + 31\times 2^{2 k} + 2^{4 k} - 9\times 2^{1 + k} - 2^{2 + k}\\
&- 11 \times 2^{3 + k} 
+  3\times  2^{3 + 2 k} - 3\times  2^{1 + 3 k} - 2^{2 + 3 k}=15\times 2^3. 
\end{align*}

For the other cases $n+1=3\times 2^k, 5\times 2^k, 15\times 2^k$, 
we obtain the solutions similarly as follows: 
\[
(n,\ell)=(0,3),(1,3),(2,5),(3,4),(4,7),(5,7),(11,11). 
\]
\end{proof}

\subsection{Proof of Theorem \ref{thm:main1} (2)}

In this subsection, 
we provide the proof of Theorem \ref{thm:main1} (2).
Let 
\[
W_{C} (x,y)=x^{n}+\alpha x^{n-d_1}y^{d_1}+\beta x^{n-d_2}y^{d_2}
+\gamma x^{\frac{n}{2}}y^{\frac{n}{2}}+\beta x^{d_2}y^{n-d_2}+\alpha x^{d_1}y^{n-d_1}+y^{n}
\]
be the weight enumerator of $C$. 
We remark that 
\[
1+\alpha+\beta+\frac{\gamma}{2}=2^{k-1}. 
\]
First, we show that 
if $d^\perp\geq 12$, then 
we have the constraints Eqs.~(\ref{eq:z2}).
By Theorem \ref{thm: macwilliams iden.},
\begin{align*}
W_{C^\perp} (x ,y)=& 2^{-k} W_C(x+y,x-y) \\
=&\sum_{i\geq 0}A_ix^{n-2i}y^{2i}.
\end{align*}
If the coefficient of $x^{n-2i}y^{2i}$ $(1\leq i\leq 5)$ 
in $W_{C^\perp} (x ,y)$ is zero, 
then 
\begin{align}\label{eq:z2}
A_2=
A_4=
A_6=
A_8=
A_{10}=0.
\end{align}

Therefore, 
if $d^\perp\geq 12$, then 
we have the constraints Eqs.~(\ref{eq:z2}). 
Using Eqs.~(\ref{eq:z2}), 
we show that the following theorem:
\begin{Thm}\label{thm:d^perp>12}
There is no code $C$ with $d^{\perp}\geq 12$.
\end{Thm}
\begin{proof}
We assume that $C$ has $d^{\perp}\geq 12$. 
Using Eqs.~(\ref{eq:z2}), we write
$\alpha,\beta$, and $\gamma$ in terms of 
$n,d_1$, and $d_2$, that is, 
\begin{align*}
\alpha=\alpha_1&=Y_{11}(n,d_1,d_2),\\
\beta=\beta_1&=Y_{12}(n,d_1,d_2),\\
\gamma=\gamma_1&=Y_{13}(n,d_1,d_2). 
\end{align*}
Similarly, 
using Eqs.~(\ref{eq:z2}), we write
$\alpha,\beta$, and $\gamma$ in terms of 
$n,d_1$, and $d_2$, that is, 
\begin{align*}
\alpha=\alpha_2&=Y_{21}(n,d_1,d_2),\\
\beta=\beta_2&=Y_{22}(n,d_1,d_2),\\
\gamma=\gamma_2&=Y_{23}(n,d_1,d_2), 
\end{align*}
and using Eqs.~(\ref{eq:z2}), we write
$\alpha,\beta$, and $\gamma$ in terms of 
$n,d_1$, and $d_2$, that is, 
\begin{align*}
\alpha=\alpha_3&=Y_{31}(n,d_1,d_2),\\
\beta=\beta_3&=Y_{32}(n,d_1,d_2),\\
\gamma=\gamma_3&=Y_{33}(n,d_1,d_2). 
\end{align*}

Using {\sc Mathematica}, we obtain the solutions of 
\[
\alpha_1=\alpha_2, \alpha_1=\alpha_3,
\beta_1=\beta_2, \beta_1=\beta_3,
\gamma_1=\gamma_2, \gamma_1=\gamma_3. 
\] 
We note that these solutions are listed on the homepage of
one of the authors \cite{miezaki}. 
Because 
$0<d_1<d_2<n/2$ and $5<n$, 
the solutions (1b)--(19b), (22b), (23b), (26b), and (27b) 
in \cite{miezaki} are impossible. 
We show that (25b) in \cite{miezaki} is impossible. 
The other cases (20b),(21b), and (24b) can be proved similarly. 

Then using the solution (25b) in \cite{miezaki} 
and Eqs.~(\ref{eq:z2}), 
using {\sc Mathematica}, we obtain 
\[
1+\alpha+\beta+\frac{\gamma}{2}
=\frac{1}{120}(184 n - 110 n^2 + 55 n^3 - 10 n^4 + n^5). 
\]
By Lemma \ref{lem:eq2}, 
\[
(n,k)=(1,0),(2,1),(3,2),(4,3),(5,4),(6,5),(12,10), 
\]
and 
it is clear that these cases 
are impossible. 
\end{proof}

\begin{proof}[Proof of Theorem \ref{thm:main1} (1)]
By Theorem \ref{thm:d^perp>12}, 
$d^\perp \leq 10$. Hence,
$(d^\perp, t) = (6, 1)$, $(8, 3)$, or $(10,5)$.
\end{proof}


\subsection{Proof of Theorem \ref{thm:main2} (2)}

In this subsection, 
we provide the proof of Theorem \ref{thm:main2} (2). 
\begin{proof}
The harmonic weight enumerator of $f\in \Harm_{t+1}$ is 
\begin{align*}
W_{C,f}
=&a_{1}x^{n-d_{1}}y^{d_{1}}+ a_{2}x^{n-d_{2}}y^{d_{2}} +bx^{\frac{n}{2}}y^{\frac{n}{2}} +c_{2}x^{d_{2}}y^{n-d_{2}} +c_{1}x^{d_{1}}y^{n-d_{1}}\\
=&(xy)^{t+1}(a_{1}x^{n-d_{1}-(t+1)}y^{d_{1}-(t+1)}+ a_{2}x^{n-d_{2}-(t+1)}y^{d_{2}-(t+1)} \\
&+bx^{\frac{n}{2}-(t+1)}y^{\frac{n}{2}-(t+1)}  +c_{2}x^{d_{2}-(t+1)}y^{n-d_{2}-(t+1)} \\
&+c_{1}x^{d_{1}-(t+1)}y^{n-d_{1}-(t+1)}), 
\end{align*}
where $a_{1},a_{2},b,c_{1},c_{2} \in \RR$. 
We set
\begin{align*}
Z_{C,f}=&a_{1}x^{n-d_{1}-(t+1)}y^{d_{1}-(t+1)}+ a_{2}x^{n-d_{2}-(t+1)}y^{d_{2}-(t+1)} \\
&+bx^{\frac{n}{2}-(t+1)}y^{\frac{n}{2}-(t+1)} \\
 &+c_{2}x^{d_{2}-(t+1)}y^{n-d_{2}-(t+1)}\\
& +c_{1}x^{d_{1}-(t+1)}y^{n-d_{1}-(t+1)}. 
\end{align*}

Then by Theorem \ref{thm: Bachoc iden.}, 
\begin{align*}
Z_{C^{\perp},f}
= &a'_{1}(x+y)^{n-d_{1}-(t+1)}(x-y)^{d_{1}-(t+1)}\\
&+ a'_{2}(x+y)^{n-d_{2}-(t+1)}(x-y)^{d_{2}-(t+1)} \\
 &+b'(x+y)^{\frac{n}{2}-(t+1)}(x-y)^{\frac{n}{2}-(t+1)} \\
 &+c'_{2}(x+y)^{d_{2}-(t+1)}(x-y)^{n-d_{2}-(t+1)}\\
& +c'_{1}(x+y)^{d_{1}-(t+1)}(x-y)^{n-d_{1}-(t+1)} \\ 
= &a'_{1}(x^{2}-y^{2})^{d_{1}-(t+1)}(x+y)^{n-2d_{1}}\\
&+ a'_{2}(x^{2}-y^{2})^{d_{2}-(t+1)}(x+y)^{n-2d_{2}} \\
 &+b'(x^{2}-y^{2})^{\frac{n}{2}-(t+1)} \\
 &+c'_{2}(x^{2}-y^{2})^{d_{2}-(t+1)}(x-y)^{n-2d_{2}} \\
&+c'_{1}(x^{2}-y^{2})^{d_{1}-(t+1)}(x-y)^{n-2d_{1}}.
\end{align*}

Since $C^{\perp}$ does not have an odd weight,
$a'_{1}-c'_{1}=0$ and $a'_{2}-c'_{2}=0$.
By $d^{\perp} \neq t+1$ and $t+3$, 
the coefficients of $x^{n-2(t+1)}$ and $x^{n-2(t+2)}y^{2}$ are zero. 
Then, 
\begin{align}
&2a'_{1}+2a'_{2}+b'=0, \label{eqn:t+1} \\
&2a'_{1} \left( d_{1}-(t+1)+ \binom{n-2d_{1}}{2} \right)+ 2a'_{2} \left( d_{2}-(t+1)+ \binom{n-2d_{2}}{2}\right) \label{eqn:t+3} \\
&+b' \left( \frac{n}{2}-(t+1) \right)=0. \notag
\end{align} 

By (\ref{eqn:t+1}) and (\ref{eqn:t+3}),  
\begin{align*}
&a'_{2}=-\frac{(n-2d_{1})(n-2d_{1}-2)}{(n-2d_{2})(n-2d_{2}-2)}a'_{1}, \\
&b'= \frac{8(d_{2}-d_{1})(n-d_{1}-d_{2}-1)}{(n-2d_{2})(n-2d_{2}-2)} a'_{1}.
\end{align*}
Then, 
\begin{align}\label{eqn:w5} 
Z_{C^{\perp},f} 
= & a'_{1} \bigg( (x^{2}-y^{2})^{d_{1}-(t+1)}(x+y)^{n-2d_{1}} \\
&-\frac{(n-2d_{1})(n-2d_{1}-2)}{(n-2d_{2})(n-2d_{2}-2)}(x^{2}-y^{2})^{d_{2}-(t+1)}(x+y)^{n-2d_{2}}\notag \\
& +\frac{8(d_{2}-d_{1})(n-d_{1}-d_{2}-1)}{(n-2d_{2})(n-2d_{2}-2)}(x^{2}-y^{2})^{\frac{n}{2}-(t+1)} \notag \\
& -\frac{(n-2d_{1})(n-2d_{1}-2)}{(n-2d_{2})(n-2d_{2}-2)}(x^{2}-y^{2})^{d_{2}-(t+1)}(x-y)^{n-2d_{2}}\notag \\
&+(x^{2}-y^{2})^{d_{1}-(t+1)}(x-y)^{n-2d_{1}} \bigg). \notag
\end{align}

Let 
\[
W_{C^\perp,f}=(xy)^{t+1}Z_{C^{\perp},f}=\sum{p_i}x^{n-i}y^i.
\]
Recall that if $p_{2w+t+1}=0$ then $D^{\perp}_{2w+t+1}$ is a $(t+1)$-design. 
By (\ref{eqn:w5}),  
\begin{align*}
p_{2w+t+1}= & {\rm (constant)}\\
\times &
\Bigg(\sum_{i=0}^{w}  (-1)^{w-i} \binom{d_{1}-(t+1)}{w-i}\binom{n-2d_{1}}{2i} \\
&-\frac{(n-2d_{1})(n-2d_{1}-2)}{(n-2d_{2})(n-2d_{2}-2)} \sum_{j=0}^{w}  (-1)^{w-j} \binom{d_{2}-(t+1)}{w-j}\binom{n-2d_{2}}{2j} \\
&+\frac{8(d_{2}-d_{1})(n-d_{1}-d_{2}-1)}{(n-2d_{2})(n-2d_{2}-2)} (-1)^{w+1} \binom{n/2-(t+1)}{w} \Bigg).
\end{align*}
By Theorem \ref{thm:design}, 
if the equation 
\begin{align*}
 &\sum_{i=0}^{w}  (-1)^{w-i} \binom{d_{1}-(t+1)}{w-i}\binom{n-2d_{1}}{2i} \\
&-\frac{(n-2d_{1})(n-2d_{1}-2)}{(n-2d_{2})(n-2d_{2}-2)} \sum_{j=0}^{w}  (-1)^{w-j} \binom{d_{2}-(t+1)}{w-j}\binom{n-2d_{2}}{2j} \\
&+\frac{8(d_{2}-d_{1})(n-d_{1}-d_{2}-1)}{(n-2d_{2})(n-2d_{2}-2)} (-1)^{w+1} \binom{n/2-(t+1)}{w} 
=0, 
\end{align*}
then $D^{\perp}_{2w+t+1}$ is a $(t+1)$-design. 
\end{proof}

\section{Concluding Remarks}\label{sec:rem}

\begin{Rem}
\begin{enumerate}

\item [(1)]


Are there examples that satisfy the 
condition of Theorem \ref{thm:main2}?

\item [(2)]

For the case $d^{\perp}-t= 4$, 
if we assume that $d^\perp\geq 10$ and 
\begin{align*}
W_{C} (x,y)&=x^{n}+\alpha x^{n-d_{1}}y^{d_{1}}+\beta x^{n-d_{2}}y^{d_{2}}\\
&+\beta x^{d_{2}}y^{n-d_{2}}+\alpha x^{d_{1}}y^{n-d_{1}}
+y^{n}, 
\end{align*}
then we have 
\[
1+\alpha+\beta=\sum_{i=0}^4\binom{n-1}{i}. 
\]
Similarly, 
for the case $d^{\perp}-t= 5$, 
if we assume that $d^\perp\geq 12$ and 
\begin{align*}
W_{C} (x,y)&=x^{n}+\alpha x^{n-d_1}y^{d_1}+\beta x^{n-d_2}y^{d_2}
+\gamma x^{\frac{n}{2}}y^{\frac{n}{2}}\\
&+\beta x^{d_2}y^{n-d_2}+\alpha x^{d_1}y^{n-d_1}+y^{n}, 
\end{align*}
then we have 
\[
1+\alpha+\beta+\frac{\gamma}{2}=\sum_{i=0}^5\binom{n-1}{i}. 
\]
This suggests the following conjecture: 
\begin{Conj}
Let $C$ be a binary antipodal $[n,k]$ code and 
\[
W_{C} (x,y)=x^{n}+\sum_{i\geq 1}\alpha_i x^{n-d_i}y^{d_i}+y^n, 
\]
be the weight enumerator of $C$.
We assume that $C$ satisfies the AM-condition with $d^{\perp}-t=\ell$ ($\ell \geq 6$).
If we assume that $d^\perp\geq 2\ell+2$, 

then we have 
\begin{align*}
\begin{cases}
\displaystyle 
1+\alpha_1+\cdots+\alpha_{\ell/2}=\sum_{i=0}^\ell\binom{n-1}{i}
=2^{k-1} \ 
({\rm if }\ \ell\equiv 0\pmod{2}), \\
\displaystyle 
1+\alpha_1+\cdots+\frac{\alpha_{\lceil\ell/2\rceil}}{2}=\sum_{i=0}^\ell\binom{n-1}{i}
=2^{k-1} \ 
({\rm if }\ \ell\equiv 1\pmod{2}).
\end{cases}
\end{align*}
Moreover, 
we have $t \leq \ell$. 
\end{Conj}
To date, we do not have a proof of this conjecture. 






\end{enumerate}

\end{Rem}


\section*{Acknowledgments}

The authors would also like to thank the anonymous reviewers 
for their beneficial comments on an earlier version of the manuscript. 
The second named author was supported by JSPS KAKENHI (22K03277). 
We thank Maxine Garcia, PhD, from Edanz (https://jp.edanz.com/ac) for editing a draft of this manuscript.



\end{document}